\documentclass[a4paper,11pt]{article}
\usepackage[x11names]{xcolor}
\usepackage{amsthm}
\usepackage{amsmath}
\usepackage{amssymb}
\usepackage[unicode=true,
bookmarks=true,bookmarksnumbered=false,bookmarksopen=false,pdfborder={0 0 1},backref=false,colorlinks=true, urlcolor=purple, citecolor=red, linkcolor=purple]
 {hyperref}
 \usepackage[margin=2.9cm]{geometry}

\usepackage{graphicx}

\DeclareGraphicsExtensions{.eps,.pdf,.png.,ps}
\renewcommand{\Re}{\operatorname{Re}}

\makeatletter
  \theoremstyle{plain}
  \newtheorem{lem}{\protect\lemmaname}
 \ifx\proof\undefined\
   \newenvironment{proof}[1][\proofname]{\par
     \normalfont\topsep6\p@\@plus6\p@\relax
     \trivlist
     \itemindent\parindent
     \item[\hskip\labelsep
           \scshape
       #1]\ignorespaces
   }{%
     \endtrivlist\@endpefalse
   }
   \providecommand{\proofname}{Proof}
 \fi
  \theoremstyle{plain}
  \newtheorem{prop}{\protect\propositionname}
  \theoremstyle{plain}
  \newtheorem{thm}{\protect\theoremname}

\makeatother

\providecommand{\lemmaname}{Lemma}
\providecommand{\propositionname}{Proposition}
\providecommand{\theoremname}{Theorem}

\begin{document}

\title{On the duality between consensus problems and Markov processes, with
application to delay systems}

\author{Fatihcan M. Atay \\Department of Mathematics, Bilkent University, 06800 Ankara, Turkey}

\date{Preprint. Published in \href{http://math-mprf.org/journal/articles/id1433}{\emph{Markov Processes and Related Fields,} 22(3):537--553, 2016.}}

\maketitle

\begin{abstract}
We consider consensus of multi-agent systems as a dual problem to
Markov processes. Based on an exchange of relevant notions and results
between the two fields, we present a uniform framework which admits
the introduction and treatment of time delays in a common setting.
We study both information propagation and information processing delays,
and for each case derive conditions for reaching consensus and calculate
the consensus value.\bigskip{}

\end{abstract}

\section{Consensus and synchronization problems}

\label{sec:intro} A class of coupled dynamical systems, sometimes
referred to as \emph{cooperative systems}, deals with individual units
that try to coordinate their actions through local interactions. One
of the simplest examples is the linear \emph{consensus problem }\cite{Ren2005survey,Olfati2007survey}\emph{,}
which can be described by a set of differential equations in continuous
time $t\in\mathbb{R}$, 
\begin{equation}
\dot{x}_{i}(t)=\sum_{j\in V}a_{ij}(x_{j}(t)-x_{i}(t)),\quad i\in V,\label{cont}
\end{equation}
or difference equations in discrete time $t\in\mathbb{N}$,
\begin{equation}
x_{i}(t+1)=x_{i}(t)+\sum_{j\in V}a_{ij}(x_{j}(t)-x_{i}(t)),\quad i\in V.\label{disc}
\end{equation}
The name is suggestive of a group of people (here indexed by $i\in V$)
negotiating to reach an agreement; however, the model appears also
in other contexts, as mentioned below. In these models $V$ is a finite
set of size $n$, so we take $V=\{1,2,\dots,n\}$, although some of
the results in this paper also extend to a countable $V$. The scalar
$x_{i}(t)$ denotes the ``opinion'' of person $i$ at time $t$,
which is updated (continuously or at regular intervals, respectively,
in \eqref{cont} and \eqref{disc}) according to the interaction scheme
given by the numbers $a_{ij}\ge0$. Thus, the individuals try to adjust
their opinion to match those of the others they interact with (their
neighbors), whereas in the absence of interaction ($a_{ij}=0$ $\forall i,j$)
they would maintain their states. The interactions may differ in strength,
measured by $a_{ij}$, may be asymmetric ($a_{ij}\neq a_{ji}$) or
unidirectional ($a_{ij}>0$ while $a_{ji}=0$)%
\footnote{The reader will notice that one could as well set $a_{ii}=0$ $\forall i$
without affecting the equations \eqref{cont}--\eqref{disc}. However,
the terms $a_{ii}$ may have significance when time delays are introduced
into the model, as in Section \ref{sec:delays}.%
}. The positivity of the $a_{ij}$ is what makes the system \emph{cooperative}
\cite{Olfati2007survey}\emph{. }Indeed, (\ref{cont})--(\ref{disc})
are based on the familiar \emph{negative feedback principle} of control
theory applied pairwise to the individuals: the discrepancies in opinions
of the neighbors act as a forcing term towards reducing the discrepancies.
Since individuals have different neighborhoods, and they may thus
all be changing their states in different ways, it is not clear whether
everybody will eventually agree on a common opinion. In fact, it is
well known from control theory that aggressive feedback schemes (i.e.,
too large values of the ``gains'' $a_{ij})$ may lead to overshoots
and oscillations instead of convergence to the desired value. We say
that the system (\ref{cont}) or (\ref{disc})\emph{ reaches} \emph{consensus}
if for any set of initial conditions there exists some $c\in\mathbb{R}$
such that $\lim_{t\to\infty}x_{i}(t)=c$ for all $i$. The number
$c$ is then called the \emph{consensus value}.

Equations (\ref{cont})--(\ref{disc}) arise in several applications
through linearization of nonlinear systems near appropriate operational
points. For example, in simple models used in microscopic traffic
dynamics, the motion of the $i$th vehicle may be described by \cite{Chandler58,Sipahi10survey}
\begin{equation}
\dot{v}_{i}(t)=\sum_{j=1}^{n}\hat{a}_{ij}h(v_{j}(t)-v_{i}(t)),\label{traffic}
\end{equation}
where $v_{i}$ denotes the speed and the left hand side is the acceleration
of vehicle $i$. In a typical single-lane car-following scenario,
the speed of vehicle $i$ is determined by the vehicle in front, $i-1$,
so that $\hat{a}_{ij}>0$ if $j=i-1$ and zero otherwise. The desired
behavior, to avoid accidents, is that all vehicles travel at the same
common speed $c$, which may vary depending on changes in road and
traffic conditions. The function $h$ is typically monotone increasing
and differentiable; furthermore $h(0)=0$ so that traveling at the
common speed $c$ is an equilibrium configuration of the system. The
stability of such configurations is found by the linear variational
equation, which is of the form (\ref{cont}) with $x_{i}(t)=v_{i}(t)-c$
and $a_{ij}=\hat{a}_{ij}h'(0)$. 

The consensus problem is also closely related to synchronization problems.
In fact, the two notions coincide in the setting of identical phase
oscillators, which is described by \cite{Kuramoto84b} 
\begin{equation}
\dot{\theta}_{i}(t)=\omega+\sum\hat{a}_{ij}h(\theta_{j}(t)-\theta_{i}(t)),\label{kuramoto}
\end{equation}
where $\theta_{i}\in S^{1}$ is the phase of the $i$th oscillator
whose intrinsic frequency is $\omega$, and $h:S^{1}\to\mathbb{R}$
has properties similar to the $h$ appearing in the traffic problem
(\ref{traffic}). A synchronous solution has the form $\theta_{i}(t)=\Omega t+c$
$\forall i$, for some common frequency $\Omega$ (in this case equal
to $\omega$) and phase $c$. The stability of such synchronous solutions
can be studied by the linear variational equation, which again has
the form (\ref{cont}) with $x_{i}(t)=\theta_{i}(t)-\Omega t-c$ and
$a_{ij}=\hat{a}_{ij}h'(0)$. Note that the procedure involves moving
to a co-rotating frame with frequency $\Omega$, which maps the synchronous
solution to an equilibrium point, thereby reducing the problem of
synchronization to one of consensus. 

Although the consensus problems described above are completely deterministic,
there are intimate relations to Markov processes. In fact, one of
the first publications in this area studied the concept of reaching
consensus in terms of probability distributions \cite{DeGroot}. We
will be exploiting the connections further in this paper. After giving
a general overview of consensus analysis for (\ref{cont})--(\ref{disc})
in Section~\ref{sec:undelayed}, we will observe connections to Markov
processes in Section~\ref{sec:markov}, which show that the two problems
correspond to dynamics evolving on dual spaces. In Section~\ref{sec:delays}
we introduce time delays into the consensus problem. Indeed, in many
real-life problems the transmission and processing of information
are not instantaneous; therefore, the models mentioned above need
to be amended to include delays. We distinguish between two main sources
of delays, namely \emph{information} \emph{propagation} and \emph{information}
\emph{processing} delays, which are separately studied in Sections~\ref{sec:propagation}
and \ref{sec:processing}, respectively. We allow the dynamics to
depend on past history in a general way by using distributed delays,
obtain conditions for convergence to consensus, and explicitly calculate
the consensus value. We conclude with some brief remarks on several
related problems.

\section{Convergence to consensus}

\label{sec:undelayed} For notation, we let $\mathcal{D}(c,r)=\{z\in\mathbb{C}:|z-c|\le r\}$
denote the closed disc in the complex plane centered at $c$ and having
radius $r$. The set of real-valued functions on the index set $V=\{1,2,\dots,n\}$
is denoted $\mathcal{V}$; it is naturally isomorphic to the vector
space $\mathbb{R}^{n}$. The span of the vector $\mathbf{1}:=(1,\dots,1)^{\top}\in\mathbb{R}^{n}$
corresponds to the subspace of constant functions on $V$. We denote
by $A=[a_{ij}]$ the matrix of interaction weights, which are assumed
to be nonnegative. Let $d_{i}=\sum_{j\in V}a_{ij}$ be the total incoming
weights to unit $i$ (the \emph{in-degree}), and $D=\mathrm{diag}(d_{1},\dots,d_{n})$
be the diagonal matrix of in-degrees. We let $\Delta=\max_{i}\{d_{i}-a_{ii}\}$.\emph{
}Denoting $x=(x_{1},x_{2},\dots,x_{n})^{\top}\in\mathbb{R}^{n}$,
\eqref{cont} can be written in vector form as
\begin{equation}
\dot{x}(t)=-Lx(t)\label{cont-vec}
\end{equation}
where $L=D-A$. We will refer to $L$ as the \emph{Laplacian matrix.
}Some basic properties of $L$ are given next. 
\begin{lem}
\label{lem:L} The Laplacian matrix $L=D-A$ satisfies the following.
\begin{enumerate}
\item $L$ has zero row sums. It always has zero as an eigenvalue, and the
corresponding eigenvector is $\mathbf{1}=(1,\dots,1)^{\top}$.
\item All eigenvalues of $L$ belong to the disc $\mathcal{D}(\Delta,\Delta):=\{z\in\mathbb{C}:|z-\Delta|\le\Delta\}$
in the complex plane. Consequently, nonzero eigenvalues of $L$ have
strictly positive real parts.
\end{enumerate}
\end{lem}
\begin{proof}
By the definition of degree, $d_{i}-\sum_{j=1}^{n}a_{ij}=0$. So,
$L$ has zero row sums and the first statement of the lemma is obtained.
The second statement follows by an application of Gershgorin's circle
theorem \cite{Horn-Johnson}, which says that the eigenvalues of $L=[\ell_{ij}]$
are contained in the union of discs $\bigcup{}_{i=1}^{n}\mathcal{D}(\ell_{ii},\sum_{j,j\neq i}|\ell_{ij}|)=\bigcup{}_{i=1}^{n}\mathcal{D}(d_{i}-a_{ii},d_{i}-a_{ii})$.
Since these discs are ordered with respect to inclusion, their union
equals the largest one, namely $\mathcal{D}(\Delta,\Delta)$. Clearly,
$\mathcal{D}(\Delta,\Delta)$ belongs to the right complex half-plane
and intersects the imaginary axis precisely at the origin.
\end{proof}
The general solution to \eqref{cont-vec} is
\begin{equation}
x(t)=e^{-L}{}^{t}x(0),\quad t\ge0.\label{cont-sol}
\end{equation}
The solutions can be expressed as sums of eigenmodes, all of which
converge to zero by Lemma \ref{lem:L}, except those corresponding
to the zero eigenvalue. Thus, the following condition will be significant
for convergence to the subspace $\mathrm{span}(\mathbf{1})$.
\begin{quote}
\textbf{(H1)} Zero is a simple eigenvalue of the Laplacian $L$.
\end{quote}
Indeed, let $\{u_{1},\dots,u_{n}\}$ be a basis for $\mathcal{V}$
consisting of the (generalized) eigenvectors of $L$, and let $\{\pi^{1},\dots,\pi^{n}\}$
be the dual basis, i.e. a basis for the dual space $\mathcal{V}^{*}$
such that $\left\langle \pi^{i},u_{j}\right\rangle =\delta_{j}^{i}$.
Thus, if $U=[\begin{array}{c|c|c}
u_{1} & \cdots & u_{n}\end{array}]$ denotes the matrix whose columns are the vectors $u_{i}$, then the
$\pi^{i}$ are the rows of $U^{-1}$. By the first statement of Lemma
\ref{lem:L} we can set $u_{1}=\mathbf{1}$ as the eigenvector corresponding
to the zero eigenvalue. We also use $\pi^{*}:=\pi^{1}$ to denote
the corresponding dual vector, i.e., the left eigenvector of $L$
related to the zero eigenvalue. If (H1) holds, then in the basis $\{u_{i}\}$
the operator $L$ has the form
\[
U^{-1}LU=\left[\begin{array}{cc}
0 & 0\\
0 & \tilde{L}
\end{array}\right]
\]
where $\tilde{L}$ is a $(n-1)\times(n-1)$ matrix all of whose eigenvalues
have strictly positive real parts. It follows that $\lim_{t\to\infty}U^{-1}e^{-Lt}U$
exists and equals the diagonal matrix 
\begin{equation}
\left[\begin{array}{cccc}
1\\
 & 0\\
 &  & \ddots\\
 &  &  & 0
\end{array}\right].\label{limiting-matrix}
\end{equation}
Therefore,
\begin{equation}
\lim_{t\to\infty}e^{-Lt}=U\left[\begin{array}{cccc}
1\\
 & 0\\
 &  & \ddots\\
 &  &  & 0
\end{array}\right]U^{-1}=\mathbf{1}\pi^{*}.\label{limit-exponential}
\end{equation}
Using in \eqref{cont-sol} gives
\begin{equation}
\lim_{t\to\infty}x(t)=\left\langle \pi^{*},x(0)\right\rangle \mathbf{1}.\label{limiting-value}
\end{equation}
Hence we conclude that, under the condition (H1), the system \eqref{cont-vec}
reaches consensus from arbitrary initial conditions $x(0)$, with
the consensus value given by $c=\left\langle \pi^{*},x(0)\right\rangle $. 

Although (H1) appears so far only as a sufficient condition for reaching
consensus, it is not hard to see that it is also a necessary condition.
Indeed, (H1) can fail in two ways: If the geometric multiplicity of
the zero eigenvalue is equal to its algebraic multiplicity and greater
than one, then there exists an additional eigenvector $u_{2}$ linearly
independent from $\mathbf{1}$ corresponding to eigenvalue zero, so
that initial conditions $x(0)$ belonging to the invariant subspace
$\mathrm{span}(u_{2})$ satisfy $\dot{x}(t)=0$ $\forall t\ge0$ and
so do not converge to $\mathrm{span}(\mathbf{1})$. On the other hand,
if the geometric multiplicity is less than the algebraic multiplicity,
then the solution operator $e^{-Lt}$ diverges, which can be seen
by considering the exponential of the Jordan block
\[
\exp\left(\left[\begin{array}{cccc}
0 & 1\\
 & 0 & \ddots\\
 &  & \ddots & 1\\
 &  &  & 0
\end{array}\right]t\right)=\left[\begin{array}{ccccc}
1 & t & t^{2}/2 & \cdots & t^{n-1}/(n-1)!\\
 & 1 & t &  & \vdots\\
 &  & 1 & \ddots & t^{2}/2\\
 &  &  & \ddots & t\\
 &  &  &  & 1
\end{array}\right].
\]

The analysis of the discrete-time system \eqref{disc} follows the
same lines, except that convergence requires an additional condition.
Hence, \eqref{disc} in vector form 
\begin{equation}
x(t+1)=(I-L)x(t)\label{disc-vec}
\end{equation}
has the solution 
\begin{equation}
x(t)=(I-L)^{t}x(0),\quad t\in\mathbb{N}.\label{disc-sol}
\end{equation}
Consider the following inequality:
\begin{quote}
\textbf{(H2) $\Delta<1$.}
\end{quote}
If (H2) is satisfied, Lemma \ref{lem:L} implies that $I-L$ has all
eigenvalues inside the unit circle except for one eigenvalue at unity,
which, under condition (H1) is a simple eigenvalue. Thus, $\lim_{t\to\infty}U^{-1}(I-L)^{t}U$
exists and is given by the matrix \eqref{limiting-matrix}, which
yields $\lim_{t\to\infty}(I-L)^{t}=\mathbf{1}\pi^{*}$. Using in \eqref{disc-sol}
yields \eqref{limiting-value}, i.e. consensus to the common value
$\left\langle \pi^{*},x(0)\right\rangle $, as before.

\section{Relation to Markov processes}

\label{sec:markov} Although the consensus and synchronization problems
given in Section \ref{sec:intro} are deterministic models, there
is a close connection to stochastic systems that we will exploit.
The next result indicates the basic relation.
\begin{prop}
\label{thm:stoc} For the Laplacian $L=D-A$, the matrix exponential
$e^{-Lt}$ is a stochastic matrix for all $t\ge0$; that is, its elements
are nonnegative and each of its rows sums to 1. Similarly, $P=I-L$
is a stochastic matrix whenever $\Delta\le1$. If moreover $\Delta<1$,
then 1 is the only eigenvalue of $P$ on the unit circle.\end{prop}
\begin{proof}
Define $P_{\varepsilon}=I-\varepsilon L$ for $\varepsilon\ge0$.
By the first statement of Lemma \ref{lem:L}, $P_{\varepsilon}$ has
row sums equal to 1. Moreover, its off-diagonal entries are $p_{ij}=\varepsilon a_{ij}\ge0$,
and the diagonal entries are 
\[
p_{ii}=1-\varepsilon(d_{i}-a_{ii})\ge1-\varepsilon\Delta,
\]
which are nonnegative if $\varepsilon\Delta\le1$. Thus, $P_{\varepsilon}$
is a stochastic matrix for $0\le\varepsilon\le\Delta^{-1}$. In particular,
taking $\varepsilon=1$ proves the second statement of the proposition.
The last statement regarding the case $\Delta<1$ was already observed
following the condition (H2). To prove the first statement, fix some
$\varepsilon\le\Delta^{-1}$ and consider 
\begin{equation}
e^{-\varepsilon Lt}=e^{(P_{e}-I)t}=e^{P_{\varepsilon}t}e^{-t}=e^{-t}\sum_{k=0}^{\infty}\frac{t^{k}}{k!}P_{\varepsilon}^{k}.\label{exp-series}
\end{equation}
Since for $t\ge0$ all terms in the summation are nonnegative, the
components $e^{-\varepsilon Lt}$ are also nonnegative. Furthermore,
\[
e^{-\varepsilon Lt}\mathbf{1}=e^{-t}\sum_{k=0}^{\infty}\frac{t^{k}}{k!}P_{\varepsilon}^{k}\mathbf{1}=e^{-t}\sum_{k=0}^{\infty}\frac{t{}^{k}}{k!}\mathbf{1}=\mathbf{1},
\]
showing that $e^{-\varepsilon Lt}$ has row sums equal to 1. Hence,
$e{}^{-\varepsilon Lt}$ is a stochastic matrix. Since $t\ge0$ is
arbitrary, the first statement of the proposition follows. 
\end{proof}
For $\Delta\le1$, the matrix $P=I-L$ can thus be seen as the transition
matrix of a Markov chain on $n$ states indexed by $V$. A probability
distribution $\pi$ on $V$ (i.e. a row vector whose components are
nonnegative and sum to 1) evolves by 
\begin{equation}
\pi(t+1)=\pi(t)P.\label{markov}
\end{equation}
The iteration rule \eqref{markov} describes an evolution on the space
$\mathcal{V}^{*}$, as a dual dynamics to the discrete-time consensus
problem \eqref{disc-vec}. Conditions for the convergence of Markov
chains are well studied, and based on the Perron-Frobenius theory,
are usually expressed as the transition matrix $P$ having the property
\emph{SIA} (stochastic, indecomposable, and aperiodic) \cite{Doob,Wolfowitz63}.
In particular, if 1 is is a simple eigenvalue of $P$ and there are
no other eigenvalues on the unit circle, then $\lim_{t\to\infty}P^{t}$
exists and equals a rank-one matrix with identical rows, namely $1\pi^{*}$,
where the common row $\pi^{*}$ satisfies $\pi^{*}=\pi^{*}P$ and
is the stationary distribution of the Markov process. Thus, $\lim_{t\to\infty}\pi(t)=\pi^{*}$
regardless of the initial distribution $\pi(0)$. Since the eigenvalue
1 for $P$ corresponds to eigenvalue 0 for $L$, the relation to the
consensus condition (H1) is clear. In particular, the left eigenvector
$\pi^{*}$ of $L$ corresponding to the zero eigenvalue, under the
normalization $\left\langle \pi^{*},\mathbf{1}\right\rangle =1$,
is the unique stationary distribution of $P$, and thus has nonnegative
entries.

The continuous-time consensus problem \eqref{cont-vec} is also naturally
related to a continuous-time Markov process $\{X(t):t\ge0\}$: Given
$L$, we write $-L=\varepsilon^{-1}(P_{\varepsilon}-I)$ where $P_{\varepsilon}=I-\varepsilon L$
and $\varepsilon>0$ is chosen sufficiently small so that $P_{\varepsilon}$
is a stochastic matrix, as in the proof of Proposition \ref{thm:stoc}.
Then $-L$ is the generator of a Markov process with transition probability
matrix $P_{\varepsilon}$ and jump rate $\varepsilon^{-1}.$ Indeed,
the semigroup of matrices defined by $P(t)=e^{-Lt}$ yield the transition
probabilities $(P(t))_{ij}=\mathbb{P}(X(t)=j\,|\, X(0)=i)$, thereby
specifying the process completely. Under the condition (H1) and by
\eqref{limit-exponential}, we have $\lim_{t\to\infty}P(t)=\mathbf{1\pi^{*}}$,
where $\pi^{*}$ is the left eigenvector of $L$ corresponding to
the zero eigenvalue (subject to the normalization $\left\langle \pi^{*},\mathbf{1}\right\rangle =1$),
or the stationary distribution for $P_{\varepsilon}$, as well as
for $P(t)$ $\forall t\ge0$, since $\pi^{*}e^{-Lt}$=$\pi^{*}$ by
\eqref{exp-series}. 

In both discrete and continuous time, the consensus value $\left\langle \pi^{*},x(0)\right\rangle =\sum_{i\in V}\pi_{i}^{*}x_{i}(0)$
given in (\ref{limiting-value}) now has a clear interpretation: \emph{The
consensus value is the mean value of the initial conditions $x_{i}(0)$
with respect to the stationary distribution $\pi^{*}$.} This is a
consequence of the fact that $\left\langle \pi^{*},x(t)\right\rangle $
is a conserved quantity of the dynamics; to wit,
\begin{equation}
\frac{d}{dt}\left\langle \pi^{*},x(t)\right\rangle =\left\langle \pi^{*},\dot{x}(t)\right\rangle =-\pi^{*}Lx(t)=0,\quad t\ge0.\label{eq:conserved}
\end{equation}
Similarly, in discrete time, $\left\langle \pi^{*},x(t)\right\rangle =\left\langle \pi^{*},x(0)\right\rangle $
$\forall t\in\mathbb{N}$. Combined with the knowledge of converging
to consensus, $x(t)\to c\mathbf{1}$, the consensus value is obtained
as $c=\left\langle \pi^{*},x(0)\right\rangle $. 

We conclude this section by giving another expression for the limit
distribution of Markov processes in terms of the Laplacian matrix,
which will be useful in the next section. The notation $\mathrm{adj}(L)$
denotes the \emph{adjugate} (or the classical adjoint) of $L$, i.e.,
the transpose of the matrix of cofactors.
\begin{lem}
\label{lem:adj} If zero is a simple eigenvalue of the Laplacian $L$,
then $\mathrm{adj}(L)=\alpha(\mathbf{1}\pi^{*})$ for some $\alpha\neq0$.\end{lem}
\begin{proof}
By assumption, $L$ has rank $n-1$ and a one-dimensional kernel,
which is spanned by $\mathbf{1}$. This implies that $\mathrm{adj}(L)\neq0$.
Furthermore, since $L\cdot\mathrm{adj}(L)=(\det L)I=0$, all columns
of $\mathrm{adj}(L)$ lie in the kernel of $L$. Thus, $\mathrm{adj}(L)=\mathbf{1}u$
for some unique row vector $u$. Similarly, considering the product
$\mathrm{adj}(L)\cdot L=0$, one sees that all rows of $\mathrm{adj}(L)$
must lie in the one-dimensional subspace spanned by the left eigenvector
of $L$ corresponding to the zero eigenvalue. Hence, $u=\alpha\pi^{*}$
for some scalar $\alpha$; so $\mathrm{adj}(L)=\alpha(\mathbf{1}\pi^{*})$.
Since $\mathrm{adj}(L)\neq0$, $\alpha$ is nonzero. 
\end{proof}

\section{Time delays and history dependence}

\label{sec:delays} There are two basic ways time delays can arise
in consensus problem (\ref{cont}). The system
\begin{equation}
\dot{x}_{i}(t)=\sum_{j\in V}a_{ij}\left(x_{j}(t-\tau)-x_{i}(t-\tau)\right)\label{proc-delay-scalar}
\end{equation}
involves \emph{information processing delays}, whereas 
\begin{equation}
\dot{x}_{i}(t)=\sum_{j\in V}a_{ij}\left(x_{j}(t-\tau)-x_{i}(t)\right)\label{prop-delay-scalar}
\end{equation}
contains \emph{information propagation delays}. In the former, the
individual $i$ requires a certain time $\tau\ge0$ for processing
the information on the state difference $x_{j}-x_{i}$ to update its
state, whereas in the latter the information requires some time $\tau$
to travel from $j$ to $i$. An example for processing delays would
be the traffic model \eqref{traffic}, where $\tau$ represents the
reaction time of the drivers \cite{SIAP07}. Examples of information
propagation delays include networks of oscillators, such as \eqref{kuramoto}. 

The models \eqref{proc-delay-scalar} and \eqref{prop-delay-scalar}
depict the dependence of the process on its present state at time
$t$ and its past state at time $t-\tau$. It is also possible to
consider a more general dependence on the history of states over the
interval $[t-\tau,t]$. Using the fact that any linear functional
on $C([-\tau,0],\mathbb{R})$ can be represented by a Stieltjes integral,
we can extend \eqref{proc-delay-scalar} to model a general dependence
on system's history:
\begin{equation}
\dot{x}_{i}(t)=\sum_{j\in V}a_{ij}\int_{-\tau}^{0}\left(x_{j}(t+\theta)-x_{i}(t+\theta)\right)\, d\eta(\theta),\label{proc-distdelay-scalar}
\end{equation}
where $\eta:[-\tau,0]\to\mathbb{R}$ is a function of bounded variation.
Such an $\eta$ is usually said to represent \emph{distributed delays.
}The special cases \eqref{proc-delay-scalar}--\eqref{prop-delay-scalar}
are obtained when $\eta$ is a Heaviside step function, corresponding
to a \emph{discrete delay}. In order to keep the cooperative character
of interactions ($a_{ij}\ge0)$, we stipulate that $\eta$ be a nondecreasing
function. In a similar way, \eqref{prop-delay-scalar} can be generalized
to
\begin{equation}
\dot{x}_{i}(t)=\sum_{j\in V}a_{ij}\left(\int_{-\tau}^{0}x_{j}(t+\theta)\, d\eta(\theta)-x_{i}(t)\right).\label{prop-distdelay-scalar}
\end{equation}
To admit consensus solutions, i.e. for $x_{i}(t)\equiv c$ $\forall i$
to satisfy \eqref{prop-distdelay-scalar}, we must additionally have
the normalization condition
\begin{equation}
\int_{-\tau}^{0}d\eta(\theta)=1.\label{normalization}
\end{equation}
In this sense $\eta$ can be thought as a probability distribution,
although the problem comes, again, from a completely deterministic
setting. The mean of the distribution $\eta$ will be referred to
as the mean delay,
\begin{equation}
\bar{\tau}:=-\int_{-\tau}^{0}\theta\, d\eta(\theta).\label{eq:meandelay}
\end{equation}

In vector form, systems \eqref{proc-distdelay-scalar}--\eqref{prop-distdelay-scalar}
can be written as

\begin{equation}
\dot{x}(t)=-L\int_{-\tau}^{0}x(t+\theta)\, d\eta(\theta)\label{proc-distdelay-vec}
\end{equation}
and
\begin{equation}
\dot{x}(t)=-Dx(t)+A\int_{-\tau}^{0}x(t+\theta)\, d\eta(\theta),\label{prop-distdelay-vec}
\end{equation}
respectively. Both \eqref{proc-distdelay-vec} and \eqref{prop-distdelay-vec}
represent infinite-dimensional systems even when $V$ is a finite
set. We let $\mathcal{C}=C([-\tau,0],\mathbb{R}^{n})$ denote the
Banach space of continuous functions mapping the interval $[-\tau,0]$
to $\mathbb{R}^{n}$, equipped with the sup norm. The Cauchy problem
for \eqref{proc-distdelay-vec} or \eqref{prop-distdelay-vec} requires
the initial condition $\phi$ be specified on the whole interval $[-\tau,0]$,
and the solution operator $T(t):\mathcal{C}\to\mathcal{C}$ maps $\phi$
to the solution at time $t$ in the sense
\[
x^{t}=T(t)\phi,\quad t\ge0,
\]
where $x^{t}\in\mathcal{C}$ denotes the function defined by $x^{t}(\theta)=x(t+\theta)$
for $s\in[-\tau,0]$. The infinite-dimensional setting requires its
own tools; in particular, the methods of Section \ref{sec:undelayed}
do not apply. Nevertheless, in the next sections we will draw upon
the foregoing ideas to carry out the analysis and give a unified picture.

\section{Information propagation delays}

\label{sec:propagation} We consider the continuous-time system 
\begin{equation}
\dot{x}(t)=-Dx(t)+A\int_{-\tau}^{0}x(t+\theta)\, d\eta(\theta).\label{prop-distdelay-vec-1}
\end{equation}
We assume that there are no delays in any self-connections; thus the
diagonal elements of $A$ are zero. The characteristic equation can
be found by seeking exponential solutions of the form $x(t)=e^{st}v$,
where $v\in\mathbb{R}^{n}$ and $s\in\mathbb{C}$ is the characteristic
value. Substituting into the equation shows that the characteristic
value $s$ satisfies 
\begin{equation}
\chi(s):=\det(sI+D-F(s)A)=0,\label{eq:chareq}
\end{equation}
where
\begin{equation}
F(s)=\int_{-\tau}^{0}e^{s\theta}\, d\eta(\theta).\label{F-Laplace}
\end{equation}
Note that $F(0)=1$ by \eqref{normalization}. Thus, $s=0$ is always
a characteristic value since $\chi(0)=\det(L)=0$. It is also easy
to see that 
\begin{equation}
F'(0)=-\bar{\tau}.\label{eq:F'(0)}
\end{equation}
The next result is essential for convergence to consensus.
\begin{prop}
\label{prop:conv} If zero is a simple eigenvalue of $L$, then $s=0$
is a simple characteristic value of \eqref{prop-distdelay-vec-1},
and all other characteristic values have negative real parts.\end{prop}
\begin{proof}
We first claim that, if $\chi(s)=0$ for some $s\in\mathbb{C}$ with
$\Re(s)\ge0$, then $s=0$. Thus suppose $\chi(s)=0$ and $\Re(s)\ge0$.
Then $-s$ is an eigenvalue of the matrix $D-F(s)A$. By Gershgorin's
theorem, the eigenvalues of $D-F(s)A$ belong to the union of discs
$\mathcal{D}(c_{i},r_{i})$, where $c_{i}=d_{i}$ and $r_{i}=|F(s)|\sum_{j,j\neq i}a_{ij}\le\sum_{j,j\neq i}a_{ij}=d_{i}$,
since $a_{ii}=0$ and $|F(s)|\le1$ by assumption. Hence the eigenvalues
are contained in $\mathcal{D}(\Delta,\Delta)\subset\{z\in\mathbb{C}:\Re(z)>0\}\cup\{0\}$.
Since it was assumed that $-s$ is an eigenvalue and $\Re(-s)\le0$,
it must be the case that $s=0$. This proves the claim. It also shows
that all nonzero roots of $\chi$ have strictly negative real parts. 

To complete the proof it remains to show that zero is a simple root
of $\chi$ if (H1) is satisfied. We establish this by evaluating $\chi'(0)$.
Recall that, by Jacobi's rule, the derivative of the determinant of
a matrix function $B(s)$ is given by
\[
\frac{d}{ds}\det B(s)=\mathrm{tr}[\mathrm{adj}B(s)\cdot B'(s)]
\]
where ``tr'' denotes the trace of a matrix. Thus, recalling \eqref{eq:F'(0)},
\begin{align*}
\chi'(0) & =\mathrm{tr}[\mathrm{adj}(L)\cdot(I+\bar{\tau}A)].\\
 & =\mathrm{tr}[\mathrm{adj}(L)\cdot(I+\bar{\tau}(D-L)]\\
 & =\mathrm{tr}[\mathrm{adj}(L)\cdot(I+\bar{\tau}D)],
\end{align*}
since $\mathrm{adj}(L)\cdot L=(\det L)I=0$. Letting $\hat{\ell}_{ii}$
denote the diagonal elements of $\mathrm{adj}(L)$, we thus have 
\[
\chi'(0)=\sum_{i=1}^{n}\hat{\ell}_{ii}(1+\bar{\tau}d_{ii}).
\]
By Lemma \ref{lem:adj}, $\hat{\ell}_{ii}=\alpha\pi_{i}^{*}$ for
some $\alpha\neq0$, which finally yields
\[
\chi'(0)=\alpha\sum_{i=1}^{n}\pi_{i}^{*}(1+\bar{\tau}d_{ii}).
\]
Notice that the sum on the right is strictly positive since all the
quantities under the summation are nonnegative, and the $\pi_{i}^{*}$
are not all zero, since $\pi^{*}$ is a probability distribution.
Thus, $\chi'(0)\neq0$, proving that 0 is a simple root of $\chi$. 
\end{proof}
Proposition \ref{prop:conv} implies that, in case zero is a simple
Laplacian eigenvalue, all solutions of \eqref{prop-distdelay-vec-1}
asymptotically approach the one-dimensional subspace spanned by $\mathbf{1}$
(now interpreted as an element of $\mathcal{C}$). This is a consequence
of the fact that all solutions of the linear equation \eqref{prop-distdelay-vec-1}
involve a factor $e^{st}$, where $s$ is a root of the characteristic
equation \eqref{eq:chareq}, which, by the Proposition \ref{prop:conv},
all have negative real parts, except for the single root $s=0$. The
latter root corresponds, by \eqref{eq:chareq} and \eqref{normalization},
to the kernel of $D-A=L$, which is spanned by $\mathbf{1}$. For
an alternative proof based on the direct solution of \eqref{prop-distdelay-vec-1},
the reader is referred to \cite{rsta13}. Note that the dynamics on
the subspace $\mathrm{span}(\mathbf{1})$ is constant, i.e., $\dot{x}(t)\equiv0$
by \eqref{prop-distdelay-vec-1}; hence the system reaches consensus.

We next establish a conserved quantity for the dynamics (cf.~\eqref{eq:conserved}
for the undelayed case). 
\begin{prop}
\label{thm:conserved} The quantity 
\[
q=\left\langle \pi^{*},x(t)+D\int_{-\tau}^{0}\int_{t+\theta}^{t}x(s)\, ds\, d\eta(\theta)\right\rangle 
\]
 is conserved under the dynamics of \eqref{prop-distdelay-vec-1},
i.e., $\frac{dq}{dt}=0$ along any solution $x(t)$ of \eqref{prop-distdelay-vec-1}.\end{prop}
\begin{proof}
By direct calculation using \eqref{prop-distdelay-vec-1} and \eqref{normalization},
\begin{multline*}
\frac{d}{dt}\left\langle \pi^{*},x(t)+D\int_{-\tau}^{0}\int_{t+\theta}^{t}x(s)\, ds\, d\eta(\theta)\right\rangle =\\
\left\langle \pi^{*},\dot{x}(t)\right\rangle +\left\langle \pi^{*},D\int_{-\tau}^{0}(x(t)-x(t+\theta))\, d\eta(\theta)\right\rangle \\
=\left\langle \pi^{*},\dot{x}(t)\right\rangle +\left\langle \pi^{*},Dx(t)\right\rangle -\left\langle \pi^{*},D\int_{-\tau}^{0}x(t+\theta))\, d\eta(\theta)\right\rangle =0,
\end{multline*}
where in the last line we have used the fact that $\pi^{*}D=\pi^{*}A$
since $\pi^{*}L=\pi^{*}(D-A)=0$. 
\end{proof}
Combining the above, we obtain our main result for consensus under
propagation delays. 
\begin{thm}
\label{thm:main} If zero is a simple eigenvalue of the Laplacian
$L$, then the system \eqref{prop-distdelay-vec-1} reaches consensus
from arbitrary initial conditions and the consensus value is given
by 
\begin{equation}
c=\frac{1}{1+\bar{\tau}\left\langle \pi^{*},\mathbf{d}\right\rangle }\left\langle \pi^{*},x(0)+D{\displaystyle \int_{-\tau}^{0}\int_{\theta}^{0}}x(s)\, ds\, d\eta(\theta)\right\rangle ,\label{c-value-distributed}
\end{equation}
where $\mathbf{d}=D\mathbf{1}=(d_{1},\dots,d_{n})^{\top}$ is the
degree vector and $\bar{\tau}$ is the mean delay.\end{thm}
\begin{proof}
By Proposition \ref{prop:conv} and the subsequent arguments, $\lim_{t\to\infty}x(t)=c\mathbf{1}$.
Using the conserved quantity $q$ given in Proposition \ref{thm:conserved}
and the continuity of the scalar product, we have
\begin{align*}
q & =\left\langle \pi^{*},x(0)+D\int_{-\tau}^{0}\int_{\theta}^{0}x(s)\, ds\, d\eta(\theta)\right\rangle \\
 & =\left\langle \pi^{*},c\mathbf{1}+D\int_{-\tau}^{0}\int_{\theta}^{0}c\mathbf{1}\, ds\, d\eta(\theta)\right\rangle \\
 & =c+c\left\langle \pi^{*},-D\mathbf{1}\int_{-\tau}^{0}\theta\, d\eta(\theta)\right\rangle \\
 & =c(1+\bar{\tau}\left\langle \pi^{*},D\mathbf{1}\right\rangle ),
\end{align*}
where we have recalled the definition \eqref{eq:meandelay} of $\bar{\tau}$.
Solving for $c$ establishes the theorem. 
\end{proof}
For the case when there are no delays ($\bar{\tau}=0$), the consensus
value given in Theorem~\ref{thm:main} obviously reduces to the one
of Section~\ref{sec:undelayed}: $c=\left\langle \pi^{*},x(0)\right\rangle $.
For the special case \eqref{prop-delay-scalar} when the delay is
a discrete delay at $-\tau$, \eqref{c-value-distributed} reduces
to
\begin{equation}
c=\frac{1}{1+\tau\left\langle \pi^{*},\mathbf{d}\right\rangle }\left\langle \pi^{*},x(0)+D{\displaystyle \int_{-\tau}^{0}}x(s)\, ds\right\rangle .\label{c-value-discrete}
\end{equation}
In both \eqref{c-value-distributed} and \eqref{c-value-discrete},
we once again have probabilistic interpretations of the consensus
value afforded by the invariant distribution $\pi^{*}$. The quantity
$\left\langle \pi^{*},\mathbf{d}\right\rangle $, for instance, is
the mean degree with respect to the invariant distribution. In general,
the consensus value $c$ represents an appropriate mean of the initial
data of the system, over an interval of length $\tau$, calculated
and properly normalized with respect to $\pi^{*}$.

\section{Information processing delays}

\label{sec:processing} We now consider the consensus problem under
information processing delays:

\begin{equation}
\dot{x}(t)=-L\int_{-\tau}^{0}x(t+\theta)\, d\eta(\theta)\label{proc-distdelay-vec-1}
\end{equation}
In contrast to previous sections, convergence to consensus for \eqref{proc-distdelay-vec-1}
is not guaranteed. For systems of the form \eqref{proc-distdelay-vec-1}
one has that stability persists for sufficiently small delays and
is typically lost for sufficiently large delays. This observation
becomes intuitively clear if one thinks in terms of particular applications.
For instance, for the traffic model \eqref{traffic} it is reasonable
to expect that large driver reaction times jeopardize stability. Mathematically,
however, finding exact boundaries between stability and instability
in terms of system parameters can be quite involved. The problem has
been studied in \cite{SIAP07,michiels-sicon09} for certain choices
of the delay distribution $\eta$. Here we will not pursue a complete
analysis but indicate some similarities and differences with Sections
\ref{sec:undelayed} and \ref{sec:propagation}.

We first note from \eqref{proc-distdelay-vec-1} that
\[
\frac{d}{dt}\left\langle \pi^{*},x(t)\right\rangle =-\pi^{*}L\int_{-\tau}^{0}x(s)\, d\eta(s)=0.
\]
Thus $\left\langle \pi^{*},x(t)\right\rangle $ is a conserved quantity
(cf. \eqref{eq:conserved}). If the system converges to consensus,
i.e., $x(t)\to c\mathbf{1}$, then 
\[
\left\langle \pi^{*},x(0)\right\rangle =\left\langle \pi^{*},x(t)\right\rangle =\left\langle \pi^{*},c\mathbf{1}\right\rangle =c.
\]
Hence, the consensus value is $c=\left\langle \pi^{*},x(0)\right\rangle $,
as in the undelayed case \eqref{limiting-value}. 

Actual convergence to consensus depends on the characteristic values
of \eqref{proc-distdelay-vec-1}, which are given by the roots of
the characteristic equation
\begin{equation}
\hat{\chi}(s)=\det(sI+F(s)L)\label{chareq2}
\end{equation}
where $F(s)$ is given by \eqref{F-Laplace} as before.
\begin{prop}
Zero is always a characteristic root of \eqref{chareq2}; furthermore,
it is a simple root if zero is a simple eigenvalue of $L$.\end{prop}
\begin{proof}
$\hat{\chi}(0)=\det(L)=0$, so zero is a characteristic root. Following
the lines of the proof of Proposition \ref{prop:conv}, we evaluate
$\hat{\chi}'(0)$ using Jacobi's formula and Lemma \ref{lem:adj}.
Hence, when zero is a simple eigenvalue of the Laplacian, 
\begin{align*}
\hat{\chi}'(0) & =\mathrm{tr}[\mathrm{adj}(L)\cdot(I+F'(0)L)]\\
 & =\mathrm{tr}[\mathrm{adj}(L)]=\alpha\neq0.
\end{align*}
Thus, zero is a simple root of the characteristic equation \eqref{chareq2}.
\end{proof}
To ensure convergence to consensus, the real parts of all the nonzero
characteristic roots must be negative. From (\ref{chareq2}) one has
that if $s$ is a characteristic root then the quantity $-\frac{s}{F(s)}$
equals some eigenvalue $\lambda_{i}$ of the Laplacian $L$. Thus,
the problem can be reduced to a set of scalar equations
\begin{equation}
\frac{s}{F(s)}=-\lambda_{i},\quad i=2,\dots,n.\label{chareq-scalar}
\end{equation}
Stability of characteristic equations of the form \eqref{chareq-scalar}
has been studied for special cases of $F$. One of the earliest treatments
is due to Hayes \cite{Hayes50}, and applies to the case of real-valued
right hand side and a discrete delay at $\tau>0$, i.e.~for $F(s)=e^{-s\tau}$.
In this case, Theorem 1 in \cite{Hayes50} implies that $\Re(s)<0$
for all solutions $s$ of \eqref{chareq-scalar} if and only if $0<\lambda_{i}\tau<\frac{\pi}{2}$.
Applying this result to the consensus problem \eqref{proc-delay-scalar},
and in view of the foregoing arguments, the following result is obtained
(see also \cite{olfati-saber-tac04}). 
\begin{thm}
\label{thm:proc} Consider the system \eqref{proc-delay-scalar} and
suppose the coupling is symmetric, i.e.~$a_{ij}=a_{ji}$ $\forall i,j$.
Assume further that zero is a simple eigenvalue of the Laplacian $L$.
Then \eqref{proc-delay-scalar} reaches consensus from arbitrary initial
conditions if and only if 
\[
0\le\tau<\frac{\pi}{2\max_{i}\{\lambda_{i}\}}.
\]
In particular, a sufficient condition for reaching consensus is
\begin{equation}
0\le\tau<\frac{\pi}{4\Delta}.\label{suff}
\end{equation}
 Moreover, the consensus value is given by $c=\left\langle \pi^{*},x(0)\right\rangle $.
\end{thm}
The sufficient condition \eqref{suff} follows from the second statement
of Lemma~\ref{lem:L}, which implies that $0\le\lambda_{i}\le2\Delta$,
$\forall i$. 

Theorem \ref{thm:proc} shows that the consensus problem for sufficiently
small processing delays are similar to the undelayed case, whereas
for large delays consensus from arbitrary initial conditions is not
possible. Similar statements have been obtained for uniform and gamma-distributed
delays with dead time \cite{SIAP07}. An interesting result is that
consensus may be obtained for distributed delays with mean $\tau$
even when it is not possible with a discrete delay at $\tau$ \cite{SIAP07}.
This stabilizing effect of delay distributions is analyzed in depth
in \cite{DCDS08}.

\section{Concluding remarks}

Consensus is a special case of synchronization: A coupled system of
dynamical variables $x_{i}$ is said to \emph{synchronize }if $\lim_{t\to\infty}|x_{i}(t)-x_{j}(t)|=0$
for all pairs $i,j$, without requiring that each $x_{i}$ necessarily
converge to a constant as in the consensus problem. The main notion
is nevertheless similar and involves the convergence of the system
to the one-dimensional subspace spanned by the vector $\mathbf{1}$.
Thus, the synchronized system exhibits solutions of the form $s(t)\mathbf{1}$
for some function $s(t)$. The form of the consensus models \eqref{cont}--\eqref{disc}
implies that $s(t)$ must be a constant. However, in other systems
where the right hand sides have a different form, say,
\begin{equation}
f(x_{i}(t))+\sum_{j\in V}a_{ij}g(x_{j}(t),x_{i}(t)),\label{general-rhs}
\end{equation}
a synchronized solution, if it exists, must obey $\dot{s}(t)=f(s(t))$
or $s(t+1)=f(s(t))$, which can be a complicated trajectory in high
dimensions, or even in one-dimension in discrete-time systems (see
\cite{Nonlin09} for a general discussion). The convergence analysis
proceeds by studying the linear variational equation around the synchronized
solution $s(t)\mathbf{1}$, which yields a \emph{non-autonomous} linear
system, as the main difference with the present paper. 

A related class of problems deals with the situation when the interaction
structure described by the $a_{ij}$ varies in time. Practical examples
involve units moving in space where obstacles may hinder communication,
failure and recovery of nodes or links in networks, changing social
relations, and so on. In such cases the Laplacian $L=L(t)$ is a time-varying
matrix, obeying some deterministic or stochastic dynamical rule. References
\cite{SIMA07,EPJ08} give a synchronization analysis in such systems
based on the probabilistic concept of Hajnal diameter \cite{Hajnal56,Hajnal58}
for infinite matrix sequences. Reference \cite{NHM11} studies discrete-time
consensus when the time dependence is driven by a Markovian jump process
and, in addition, the time delays are allowed to vary across the network
and over time, i.e., $\tau=\tau_{ij}(t)$. Similar to the spirit of
the present paper, these works provide further examples of employing
notions from stochastic analysis for the study of deterministic systems.

\subsection*{Acknowledgement}

This work was supported by the European Union within its 7th Framework
Programme under grant agreement \#318723 (MatheMACS).


\end{document}